\documentclass[a4paper,11pt]{article}

  \renewcommand\appendix{\par
  \setcounter{section}{0}
  \setcounter{subsection}{0}
  \setcounter{figure}{0}
  \setcounter{table}{0}
  \renewcommand\thesection{ Appendix \Alph{section}}
  \renewcommand\thefigure{\Alph{section}\arabic{figure}}
  \renewcommand\thetable{\Alph{section}\arabic{table}}
}

 \usepackage{amsmath, amsthm, amssymb}
\usepackage{amssymb}
\usepackage{graphicx}
\usepackage{algorithmic}
\usepackage{algorithm}

\usepackage{tikz}
\usetikzlibrary{calc}
\usetikzlibrary{patterns}
\tikzstyle{mybox} = [draw=black, fill=white,  thick,
    rectangle, inner sep=10pt, inner ysep=20pt]
\tikzstyle{mybox} = [draw=black, fill=white,  thick,
    rectangle, inner sep=2pt, inner ysep=2pt]

\usepackage[margin=1.2in]{geometry}

\usepackage{setspace}

\newtheorem{definition}{Definition}[section]
\newtheorem{lemma}{Lemma}[section]
\newtheorem{theorem}{Theorem}[section]

\newtheorem{remark}{Remark}[section]

\begin{document}

\title{Lower bounds for the chromatic number of certain Kneser-type  hypergraphs}
\author{Soheil Azarpendar\\ Amir Jafari}

\maketitle

\begin{abstract}
 Let $n\ge 1$, $r\ge 2$, and $s\ge 0$ be integers and ${\cal P}=\{P_1,\dots, P_l\}$ be a partition of $[n]=\{1,\dots, n\}$ with $|P_i|\le r$ for $i=1,\dots, l$.  Also, let $\cal F$ be a family of non-empty subsets of $[n]$. The $r$-uniform Kneser-type hypergraph $\mbox{KG}^r({\cal  F}, {\cal P},s)$ is the hypergraph with the vertex set of all $\cal P$-admissible elements $F\in {\cal F}$, that is $|F\cap P_i|\le 1$ for $i=1,\dots, l$ and the edge set of all $r$-subsets $\{F_1,\dots, F_r\}$ of the vertex set that $|F_i\cap F_j|\le s$ for all $1\le i<j\le r$. In this article, we extend the equitable $r$-colorability defect $\mbox{ecd}^r({\cal F})$ of  Abyazi Sani and Alishahi to the case when one allows intersection among the vertices of an edge. It will be denoted by $\mbox{ecd}^r({\cal F},s)$. We then, give (under certain assumptions) lower bounds for the chromatic number of $\mbox{KG}^r({\cal F}, {\cal P},s)$ and some of its variants in terms of $\mbox{ecd}^r({\cal F},\lfloor s/2\rfloor)$. This work generalizes many existing results in the literature of the Kneser hypergraphs.  It generalizes the previous results of the current authors from the special family of all $k$-subsets of $[n]$ to a general family $\cal F$ of subsets.
 \end{abstract}

\section{Introduction}

Let $n\ge 1$ and $r\ge 2$ be integers and ${\cal P}=\{P_1,\dots, P_l\}$ be a partition of $[n]=\{1,\dots, n\}$ with $|P_i|\le r$ for $i=1,\dots, l$. Let $\cal F$ be a family of non-empty subsets in $[n]$. The $r$-uniform Kneser-type hypergraph $\mbox{KG}^r({\cal  F}, {\cal P})$ is the hypergraph with the vertex set of all $\cal P$-admissible elements $F\in {\cal F}$, that is $|F\cap P_i|\le 1$ for $i=1,\dots, l,$ and the edge set of all $r$-subsets $\{F_1,\dots, F_r\}$ of the vertex set that are pairwise disjoint. This hypergraph first was considered by Alishahi and Hajiabolhassan in \cite{AH}. It was later considered by Aslam, Chen, Coldren, Frick, and  Setiabrata in \cite{F}. For an integer $s\ge 0$, that is assumed to have the property\footnote{Without this assumption, we will have a loop edge $\{F,\dots, F\}$ and the chromatic number of the associated hypergraph is by convention infinity, so there is no need to give a lower bound.} $s<|F|$ for all elements $F$ of $\cal F$, we may relax the condition of pairwise disjointness to $|F_i\cap F_j|\le s$ and arrive at the definition of the $r$-uniform hypergraph  $\mbox{KG}^r({\cal F},{\cal P}, s)$. 
 We are interested here to find lower bounds for the chromatic number $\chi(\mbox{KG}^r({\cal F},{\cal P},s))$ of this hypergraph in terms of a generalization of the equitable $r$-colorability defect of Abyazi Sani and Alishahi \cite{AA}. This result is an extension of the previous results of the current authors in \cite{AJ}.

An equitable partition of a set $X$ is a partition of it into subsets $X_i$ for $i=1,\dots, r$ such that $||X_i|-|X_j||\le 1$ for all $1\le i\le j\le r$. The equitable $r$-colorability defect $\mbox{ecd}^r({\cal F})$ of a family $\cal F$ of non-empty subsets in $[n]$ is the minimum size of a subset $X_0\subseteq [n]$ so that there is an equitable partition
$$[n]\backslash X_0=X_1\cup\dots\cup X_r$$
with the property that there are no elements $F\in {\cal F}$ and $i=1,\dots, r,$ with $F\subseteq X_i$. Abyazi Sani and Alishahi \cite{AA}  proved the following generalization of the corresponding result of Kriz in \cite{Kr} and \cite{Kr'} for the $r$-colorability defect.

\begin{theorem}
One has
$$\chi(\mbox{KG}^r({\cal F}))\ge \left\lceil\frac{\mbox{ecd}^r({\cal F})}{r-1}\right\rceil.$$
\end{theorem}

Here $\mbox{KG}^r({\cal F})$ is the hypergraph with no partition condition, in other words, $\cal P$ is the partition of $[n]$ by the singletons.  Our goal here, is to extend this result to the cases $\mbox{KG}^r({\cal F},{\cal P})$ and $\mbox{KG}^r({\cal F},{\cal P},s)$.

For $s\ge 0$ and subsets $A$ and $B$ of $[n]$, we write $A\subseteq_s B$ if there is a set $E$ of size at most $s$, such that $A\backslash E\subseteq B$. The general $r$-equitable colorability defect $\mbox{ecd}^r({\cal F},s)$ is the minimum size of a subset $X_0\subseteq [n]$, such that there is an equitable partition 
$$[n]\backslash X_0=X_1\cup\dots \cup X_r$$
with the property that there are no $F\in{\cal F}$ and $i=1,\dots, r$, such that $F\subseteq_s X_i$.

\begin{remark}\label{rem0}
Note that $\mbox{ecd}^r({\cal F},0)$ is the original equitable $r$-colorability defect of Abyazi Sani and Alishahi. It is easy to see that for the family of all $k$-subsets of $[n]$, denoted by  $[n]\choose k$, when $n\ge r(k-1)+1$ and $0\le s<k$, one has
$$\mbox{ecd}^r({[n]\choose k},s)=n-r(k-s-1).$$
\end{remark}

We have the following results. 

\begin{theorem}\label{thm1}
Under the above condition on the partition $\cal P$, one has
$$\chi(\mbox{KG}^r({\cal F},{\cal P}))\ge \left\lceil\frac{\mbox{ecd}^r({\cal F})}{r-1}\right\rceil.$$
\end{theorem}

\begin{theorem}\label{thm2}
Under the above condition on $s$, one has 
$$\chi(\mbox{KG}^r({\cal F},s))\ge \left\lceil\frac{\mbox{ecd}^r({\cal F},\lfloor s/2 \rfloor)}{r-1}\right\rceil.$$
Here the partition is understood to be trivial, in other words, by the singletons. 
\end{theorem}

Unfortunately to give a unified theorem that deals with the case of $\mbox{KG}^r({\cal F},{\cal P},s)$ we need to either assume that the pair $({\cal F},{\cal P})$ satisfies an extra condition or modify the definition of the hypergraph into $\widetilde{\mbox{KG}}^r({\cal F},{\cal P},s)$ as follows.
\\
\begin{definition}
The pair $({\cal F},{\cal P})$ is said to be $s$-good, if for any $\cal P$-admissible subset $A$ for which there exists $F\in \cal F$ so that $F\subseteq_s A$, one can find a $\cal P$-admissible element $F'\in {\cal F}$ such that $F'\subseteq_s A$.
\end{definition}

\begin{remark}\label{rem1}
Let us show that the pair $({[n]\choose k}, {\cal P})$ is $s$-good, if $n\ge r(k-1)+1$, $0\le s<k$, and $|P_i|\le r$. Note that by the assumption on $n$, we have at least $k$ non-empty partition parts in $\cal P$. Now suppose $F\subseteq_s A$ for a $\cal P$-admissible subset $A$ and a $k$-subset $F$. If $|A|\ge k$ any $k$-subset $F'$ of $A$ is $\cal P$-admissible. Hence we may assume, $k-s\le |A|\le k$ and therefore one can always add at most $s$ elements to $A$ from different partition parts with empty intersection with $A$, to make it into a $\cal P$-admissible $k$-subset $F'$ with $F'\subseteq_s A'$.
\end{remark}

Without any assumptions on the partition and the family, we need to modify the definition of the hypergraph $\mbox{KG}^r({\cal F},{\cal P},s)$ as follows. 

\begin{definition}
We let $\widetilde{\mbox{KG}}^r({\cal F},{\cal P},s)$ be the $r$-uniform hypergraph with the vertex set of all elements  $A$ of $\cal F$ such that
$$\sum_{i=1}^l \max{\{|A\cap P_i|-1,0\}}\le \lfloor s/2 \rfloor$$
 and the edge set of all $r$-subsets $\{A_1,\dots, A_r\}$ of vertices such that $|A_i\cap A_j|\le s$ for all $1\le i<j\le r$.
 \end{definition}
  Note that when $s=0$, the above condition is the same as $\cal P$-admissibility. Also if $\cal P$ is the trivial partition into singletons, this condition holds for all $A\in \cal F$.
  \\
 We have the following two results.
\begin{theorem}\label{main}
Under the above assumptions, one has
$$\chi(\widetilde{\mbox{KG}}^r({\cal F},{\cal P},s))\ge \left\lceil\frac{\mbox{ecd}^r({\cal F},\lfloor s/2 \rfloor)}{r-1}\right\rceil.$$
\end{theorem}

\begin{theorem}\label{main'}
If the pair $({\cal F},{\cal P})$ is $\lfloor s/2\rfloor$-good, then 
$$\chi({\mbox{KG}}^r({\cal F},{\cal P},s))\ge \left\lceil\frac{\mbox{ecd}^r({\cal F},\lfloor s/2 \rfloor)}{r-1}\right\rceil.$$
\end{theorem}
Note that theorem \ref{main} implies as its special cases, theorems \ref{thm1} and \ref{thm2}.
\begin{remark}
In \cite{Da}, Daneshpajouh presents the following lower bound for the chromatic number of the hypergraph $\mbox{KG}^r({[n]\choose k},s)$, when $0\le s<k$ and $n\ge r(k-1)+1$
$$\chi(\mbox{KG}^r({[n]\choose k},s))\ge\left\lceil \frac{n-r(k-s-1)}{r-1}\right\rceil.$$
When $n\ge r(k-1)+1$, then $\mbox{ecd}^r({[n]\choose k},s)=n-r(k-s-1)$ and hence, this is a stronger lower bound than the one obtained from Theorem \ref{thm2}. It is feasible that the above theorems remain true if one replaces $\lfloor s/2\rfloor$ with $s$.
\end{remark}

{\textbf{Acknowledgement.}} It is a great pleasure to thank Saeed Shaebani, whose careful reading of the first draft of this article and his precise comments, saved the authors from some blunders and improved the quality of this text. 
\section{Reduction of Theorem \ref{main} and Theorem \ref{main'}}

In this section, we prove the following lemma, which reduces the proof of Theorem \ref{main} and Theorem \ref{main'}, to the case when $r$ is a prime number. The proof is obtained by imitating a method used by Kriz in \cite{Kr'}, who himself followed a similar method used by Alon, Frankl, and Lov\'asz in \cite{AFL}.

\begin{lemma}\label{ind}
If Theorem \ref{main} (resp. Theorem \ref{main'}) is true for $r=r_1$ and $r=r_2$ then Theorem \ref{main} (resp. Theorem \ref{main'}) is true for $r=r_1r_2$.
\end{lemma}

\begin{proof}
Let $s'=\lfloor s/2\rfloor$ and ${\cal P}=\{P_1,\dots, P_l\}$ be a partition of $[n]$ with $|P_i|\le  r_1r_2$.  Also, let ${\cal P}'=\{P'_1,\dots, P'_{l'}\}$ be a partition obtained from $\cal P$ by partitioning each $P_i$ into at most $r_1$ pieces of sizes less than or equal to $r_2$. For $X\subseteq [n]$, define

$${\cal F}(X,s)=\{A\subseteq X\:|\:\mbox{There exists}\: F\in{\cal F}\:\mbox{such that}\:\:A\subseteq F\subseteq_s A\}.$$

We also define a new family
$${\cal F}'=\{X\subseteq [n]\; | \quad \mbox{ecd}^{r_1}({\cal F}(X,s'))>(r_1-1)t\}$$
where $t=\chi(\widetilde{\mbox{KG}}^{r_1r_2}({\cal F},{\cal P},s))$ (resp. $t=\chi({\mbox{KG}}^{r_1r_2}({\cal F},{\cal P},s))$) and let $c$ be a proper coloring of its vertices into $\{1,\dots, t\}$. Suppose $X\in{\cal F}'$ is a vertex of $\mbox{KG}^{r_2}({\cal F}',{\cal P}')$, then for each $P_i\in {\cal P}$, one has $|X\cap P_i|\le r_1$ so ${\cal P}|_X := \{P_1\cap X,\dots, P_l\cap X\}$ is  a partition of $X$ with each piece of size at most $r_1$. By the hypothesis of the lemma, for such an $X$, $\chi(\widetilde{\mbox{KG}}^{r_1}({\cal F}(X,s'), {\cal P}|_X, 0))>t$. The induced coloring $c_0$ on the ${\cal P}$-admissible elements $A\in {\cal F}(X,s')$ is defined as follows. According to the definition, let $F\in \cal F$ be such that $F\subseteq_{s'} A$, then $F$ is a vertex of $\widetilde{\mbox{KG}}^{r_1r_2}({\cal F},{\cal P},s)$ and define $c_0(A)=c(F)$, in the first case. In the case of Theorem \ref{main'}, by the goodness assumption on the pair $({\cal F},{\cal P})$, we can find a $\cal P$-admissible $F'$ such that $F'\subseteq_{s'} A$ and define $c_0(A)=c(F')$.
\\
Since $c_0$ is not a proper coloring, it follows that one may find vertices $$B_1(X),\dots, B_{r_1}(X)$$ of $\widetilde{\mbox{KG}}^{r_1}({\cal F}(X,s'),{\cal P}|_X,0)$ that are pairwise disjoint and have the same color. Define the coloring $c'$ for 
$\mbox{KG}^{r_2}({\cal F}',{\cal P}')$ by $c'(X)=c_0(B_1(X))$.  Note that for each $B_i(X)$, one has a vertex $F_i(X)$ of $\widetilde{{\mbox{KG}}}^{r_1r_2}({\cal F},{\cal P},s)$ (resp. of $ {\mbox{KG}}^{r_1r_2}({\cal F},{\cal P},s)$) such that $F_i(X)\subseteq_{s'} B_i(X)$, with $c(F_1(X))=\dots=c(F_{r_1}(X))$. Then $c'$ is a proper coloring since otherwise there exists pairwise disjoint vertices $X_1,\dots , X_{r_2}$ with the same color, and hence the $r_1r_2$ subsets $F_i(X_j)$ of $\cal F$ have pairwise intersection of at most $s$ elements and the same color. This contradicts the properness of the coloring $c$. So by the hypothesis of the lemma, $\mbox{ecd}^{r_2}({\cal F}')\le (r_2-1)t$. Hence, one may find $X_0\subseteq [n]$ of size at most $(r_2-1)t$ and an equitable partition
$$[n]\backslash X_0=X_1\cup\dots\cup X_{r_2}$$
with the property that no $X\in {\cal F}'$ is a subset of one of $X_1,\dots, X_{r_2}$. So in particular for $1\le i \le r_2$, $X_i\not\in {\cal F}'$ and hence $\mbox{ecd}^{r_1}({\cal F}(X_i,s'))\le (r_1-1)t$. This implies the existence of a subset $X_{i,0}\subseteq X_i$ of size at most $(r_1-1)t$ and an equitable partition 
$$X_i\backslash X_{i,0}=X_{i,1}\cup\dots \cup X_{i,r_1}$$
such that no $A\in {\cal F}(X_i,s')$ is a subset of one of $X_{i,1},\dots, X_{i,r_1}$. We may assume that $|X_{i,0}|=(r_1-1)t$, since if $|X_{i,0}|<(r_1-1)t$, remove one element from an $X_{i,j}$ for $j=1,\dots, r_1$ with the largest size and add it to the $X_{i,0}$ without violating any of the conditions. By repeating this process, we may assume $|X_{i,0}|=(r_1-1)t$ for $i=1,\dots, r_2$. If now $|X_{i,j}|-|X_{i',j'}|>1$ for some $1\le i,i'\le r_2$ and $1\le j,j'\le r_1$, then it follows that $|X_i|-|X_{i'}|>1$, which is a contradiction. The reason for this, is that if we let $a=|X_{i,j}|$, the minimum size that $X_i$ can have is $a+(r_1-1)(a-1)+t(r_1-1)$, and the maximum size that $X_{i'}$ can have is $a-2+(r_1-1)(a-1)+(r_1-1)t$.

 It follows that we have an equitable partition
\[[n]\backslash X_0'= X_{1,1}\cup\dots\cup X_{1,r_1}\cup \dots\cup X_{r_2,1}\cup\dots \cup X_{r_2,r_1}\]
 where 
\[X_0'=X_0\cup X_{1,0}\dots\cup X_{r_2,0}\]
is of size at most
\[(r_2-1)t+r_2(r_1-1)t=(r_1r_2-1)t\]
and this partition has the property that is no $F\in {\cal F}$ such that $F\subseteq_{s'} X_{i,j}$ for some $i=1,\dots, r_2$ and $j=1,\dots, r_1$. Since otherwise, $A=F\cap X_{i,j}\in {\cal F}(X_i,s')$ and $A\subseteq X_{i,j}$, which is a contradiction.This shows that $\mbox{ecd}^{r_1r_2}({\cal F},s')$ is less than or equal to $(r_1r_2-1)t$ or in other words, $t$ is greater than or equal to $\frac{\mbox{ecd}^{r_1r_2}({\cal F},s')}{r_1r_2-1}$. This proves the lemma.

\end{proof}

\section{Proof of Theorem \ref{main} and Theorem \ref{main'}}

To prove Theorems \ref{main} and \ref{main'}, hence we may suppose that $r=p$ is a prime number. We use $\mathbb Z_p$-Tucker lemma. We recall its statement from \cite{M}. The simplicial complex $\mbox{E}_{n-1}(\mathbb Z_p)$ has $\mathbb Z_p\times [n]$ as its vertices and all subsets $A\subseteq {\mathbb Z_p}\times [n]$  with pairwise different second components as faces. It has a free action of $\mathbb Z_p$ that acts on the first component of each vertex  by multiplication. We take $\mathbb Z_p$ to be the multiplicative group of all $p$th roots of unity.

\begin{lemma}\label{tucker}{\textbf{ ($\mathbb Z_p$-Tucker Lemma)}} Let $n, m>0$ and $m\ge \alpha\ge 0$ be integers and $p$ be a prime number. If $\lambda$ is a map from the non-empty faces of $\mbox{E}_{n-1}(\mathbb Z_p)$ to $\mathbb Z_p\times [m]$ with $\lambda(A)=(\lambda_1(A),\lambda_2(A))\in \mathbb Z_p\times [m]$ that satisfies the following properties,
\begin{enumerate}
\item If $\omega\in \mathbb Z_p$ and $A$ is a non-empty face of $E_{n-1}(\mathbb Z_p)$, then $\lambda_1(\omega\cdot A)=\omega\cdot \lambda_1(A)$ and $\lambda_2(\omega\cdot A)=\lambda_2(A)$. That is $\lambda$ is $\mathbb Z_p$-equivariant.
\item If $A_1\subseteq A_2$ be non-empty faces of $E_{n-1}(\mathbb Z_p)$ and $\lambda_2(A_1)=\lambda_2(A_2)\le \alpha$ then $\lambda_1(A_1)=\lambda_1(A_2)$.
\item If $A_1\subseteq \dots\subseteq A_p$ be non-empty faces of $E_{n-1}(\mathbb Z_p)$ and $\lambda_2(A_1)=\dots=\lambda_2(A_p)>\alpha$ then $\lambda_1(A_1),\dots, \lambda_1(A_p)$ are not pairwise distinct.
\end{enumerate}
then $$\alpha+(m-\alpha)(p-1)\ge n.$$
\end{lemma}
Now let us present our proof for Theorem \ref{main} (resp. Theorem \ref{main'}).

\begin{proof}

Let $t=\chi(\widetilde{\mbox{KG}}^p({\cal F},{\cal P},s))$ (resp. $t=\chi({\mbox{KG}}^p({\cal F},{\cal P},s))$ for the case of Theorem \ref{main'}) and let $c$ be a coloring of the vertices of this hypergraph with colors $\{1,\dots, t\}$. Let  $s'=\lfloor s/2\rfloor$, $\alpha=n-\mbox{ecd}^p({\cal F},s')$ and $m=\alpha+t$. 
Also for simplicity choose a complete  ordering on non-empty subsets of $[n]$, that has the property that if $|A|<|B|$ then $A<B$.

We define a $\mathbb Z_p$-equivariant map $\lambda$ from the non-empty faces of $\mbox{E}_{n-1}(\mathbb Z_p)$ to ${\mathbb Z}_p\times [m]$ that satisfies the two properties of the $\mathbb Z_p$-Tucker lemma and hence
$$\alpha+(m-\alpha)(p-1)=n-\mbox{ecd}^p({\cal F},s')+(p-1)t\ge n$$
and hence the result follows.  For a non-empty face $A$ of $\mbox{E}_{n-1}(\mathbb Z_p)$ and $i\in \mathbb Z_p$, let $A^i=\{1\le j\le n| (i,j)\in A\}$. The definition of $\lambda(A)=(\lambda_1(A),\lambda_2(A))\in \mathbb Z_p\times [m]$ is given in two cases.
\\
{\textbf{Case 1:}} If there is an element $F\in \cal F$ with $F\subseteq_{s'} A^i$ for some $i\in \mathbb Z_p$ and $$\sum_{i=1}^l \max{\{|F\cap P_j|-1,0\}}\le s'$$ (resp. $F$ is $\cal P$-admissible) for all $1\le j\le l$, then choose the smallest 
such subset with respect to the complete ordering on subsets of $[n]$, say $F\subseteq_{s'} A^i$ and define $$\lambda(A)=(i, c(F)+\alpha).$$ 
We remark that since $|F|>s$, one can not have more than one $i\in \mathbb Z_p$ that $F\subseteq_{s'} A^i$.
\\
{\textbf{Case 2:}} Otherwise, choose a non-empty subset $B\subseteq A$ such that for all $i\in \mathbb Z_p$ and $j=1,\dots, l$, $|B^i\cap P_j|\le 1$ and $\pi_2(B)$ is maximum with respect to the chosen complete order on subsets of $[n]$, this is clearly unique. Here $\pi_2:\mathbb Z_p\times [n]\rightarrow [n]$ is the projection onto the second component. Also, assume that $$|B^{i_1}|=\dots=|B^{i_h}|<|B^{i_{h+1}}|\le\dots \le |B^{i_p}|$$ for some $1\le h\le p$, where $h=p$ means that all the sizes are equal. Define $$\lambda_2(A)=p|B^{i_1}|+p-h.$$ Note that $\lambda_2(A)\le \alpha$. This is because by removing elements from $B^{i_{h+1}}.\dots, B^{i_p}$ (if there are any)  arbitrarily, we may assume that their sizes are $|B^{i_1}|+1$ to arrive at an equitable partition of a set of size $\lambda_2(A)$. If $\lambda_2(A)$ is greater than $n-\mbox{ecd}^p({\cal F},s')$, then by the definition of $\mbox{ecd}^p({\cal F},s')$ there is an element $F\in {\cal F}$ with $F\subseteq_{s'} B^{i_k}$ for some $k=1,\dots, p,$ and therefore
$$\sum_{i=1}^l \max{\{|F\cap P_j|-1,0\}}\le s'$$
 (resp. $F$ can be chosen so that it is $\cal P$-admissible by the $s'$-goodness assumption). This contradicts the fact that we are in the Case 2. 

The definition of  $\lambda_1(A)$ is more delicate. We define it in several sub-cases.
\\
 {\textbf{Case 2.1:}} If $h<p$, find $1\le h'<p$ such that $hh'\equiv 1\mod p$ and define 
$$\lambda_1(A)=(i_1\dots i_h)^{h'}.$$
{\textbf{Case 2.2:}} If $h=p$, find the smallest $1\le j\le l$ that $\pi_2(B)\cap P_j$ is non-empty, and take the unique subset $B'\subseteq B$ such that $\pi_2(B')=\pi_2(B)\cap P_j$. Let $\pi_1(B')=\{j_1,\dots, j_k\}$, where $\pi_1$ is the projection onto the first component. Then we have again two sub-cases:
\\
 {\textbf{Case 2.2.1:}} If $k<p$, choose $1\le k'<p$ such that $kk'\equiv 1\mod p$ and define:
$$\lambda_1(A)=(j_1\dots j_k)^{k'}.$$
 {\textbf{Case 2.2.2:}} If $k=p$, define $\lambda_1(A)$ to be the first component of the element of $B'$ with the smallest second component.

It remains to check the properties of the $\mathbb Z_p$-Tucker lemma. First, $\lambda$ is $\mathbb Z_p$-equivariant in the Case 1.  That is $\lambda_1(\omega \cdot A)=\omega\cdot \lambda_1(A)$ and $\lambda_2(\omega \cdot A)=\lambda_2(A)$ for any $\omega\in \mathbb Z_p$. This is because, if $F\subseteq A^i$ is the required subset for $A$ in case one then $F\subseteq (\omega A)^{\omega\cdot i}$ is the required subset for $\omega\cdot A$. 
\\
If $A_1\subseteq \dots\subseteq A_p$ is a chain of non-empty faces of $\mbox{E}_{n-1}(\mathbb Z_p)$ with $\lambda_2(A_1)=\dots=\lambda_2(A_p)>\alpha$, then we are in the Case 1. Hence with have vertices $F_1,\dots, F_p$ of $\mbox{KG}^p({\cal F},{\cal P},s)$ with $F_i\subseteq_{s'} A_i^{\lambda_1(A_i)}$ with $c(F_1)=\dots=c(F_p)$.  If $\lambda_1(A_1),\dots, \lambda_1(A_p)$ are pairwise distinct, then since $A_i^{\lambda_1(A_i)}\cap A_j^{\lambda_1(A_j)}=\emptyset$ for $i\ne j$ then $|F_i\cap F_j|\le 2s'\le s$ and  $\{F_1,\dots,F_p\}$ will be a mono-chromic edge, which contradicts properness of $c$. Hence the third condition of the $\mathbb Z_p$-Tucker lemma holds.
\\
To show that $\lambda$ is $\mathbb Z_p$-equivariant in Case 2, note that if $B\subseteq A$ is the required set  for $A$, then $\omega\cdot B_1\subseteq \omega\cdot B_2$ is the required set in for $\omega\cdot A$, hence $\lambda_2(A)=\lambda_2(\omega\cdot A)$. Also, the corresponding $\{i_1,\dots, i_h\}$ will be $\{\omega\cdot i_1,\dots, \omega\cdot i_h\}$. In the Case 2.1, we have
$$\lambda_1(\omega\cdot A)=((\omega\cdot i_1)\dots (\omega\cdot i_h))^{h'}=\omega^{hh'}\cdot (i_1\dots i_h)^{h'}=\omega\cdot \lambda_1(A).$$
In Case 2.2, we have $\omega\cdot B'$ as the corresponding set for $\omega\cdot A$. So in both Cases 2.2.1 and 2.2.2 it follows that $\lambda_1(\omega\cdot A)=\omega\cdot \lambda_1(A)$. This proves that $\lambda$ is $\mathbb Z_p$-equivariant.
\\
If $A_1\subseteq A_2$ are non-empty faces of $\mbox{E}_{n-1}(\mathbb Z_p)$ with $\lambda_2(A_1)=\lambda_2(A_2)\le \alpha$, then we are in the second case. With maximal subsets $B_1\subseteq A_1$ and $B_2\subseteq A_2$. Assume that
$$|B_1^{i_1}|=\dots=|B_1^{i_h}|<|B_1^{i_{h+1}}|\le\dots \le |B_1^{i_p}|$$
$$|B_2^{j_1}|=\dots=|B_2^{j_k}|<|B^{j_{k+1}}|\le\dots \le |B_2^{j_p}|$$
for some $1\le h\le p$ and $1\le k\le p$. If $\lambda_2(A_1)=\lambda_2(A_2)$, then $|B_1^{i_1}|=|B_2^{j_1}|$ and $h=k$. Now since $B_1\subseteq A_1\subseteq A_2$, by the maximality of $B_2$, we have $|B_1^i|\le |B_2^i|$. Therefore $\{i_1,\dots, i_h\}=\{j_1,\dots, j_h\}$. So in Case 2.1 we must have $\lambda_1(A_1)=\lambda_1(A_2)$.
\\
If we are in Case 2.2, then $|B_1^i|=|B_2^i|$ for all $i\in {\mathbb Z}_p$ and hence $|B_1|=|B_2|$. This implies that the first $1\le j\le l$ that $\pi_2(B_1)\cap P_j$ is non-empty is the same as the first $1\le j'\le l$ that $\pi_2(B_2)\cap P_{j'}$ is non-empty. So by the maximality and equality of $|B_1|=|B_2|$, it follows that $\pi_1(B_1')=\pi_1(B_2')$. In the Case 2.2.1 therefore $\lambda_1(A_1)=\lambda_1(A_2)$. Finally, in the Case 2.2.2 since $|P_j|\le p$, it follows that $B_1'=B_2'$ and hence the first component of the element with the smallest second component in both of them are the same, that is $\lambda_1(A_1)=\lambda_1(A_2)$. This finishes checking the conditions and hence the proof of the theorem is complete. 

\end{proof}
\section{A generalization of a theorem of Abyazi Sani and Alishahi}

In this section, using Theorem \ref{thm1}, we generalize Theorem 3 of Abyazi Sani and Alishahi in \cite{AA}.  For an integer vector $S=(s_1,\dots,s_n)$ with $0\le s_i\le r$, the notion of an $S$-disjoint  multi-set $\{A_1,\dots, A_r\}$ of subsets of $[n]$ was considered by Sarkaria and Ziegler in \cite{S}, \cite{Z}, and \cite{Z'}. It means that for all $1\le i\le n$, the number of $1\le j\le r$ that $i\in A_j$ is at most $s_i$. This generalizes the notion of pairwise disjoint that is just $S=(1,1,\dots, 1)$-disjoint. Ziegler \cite{Z} extended the $r$-colorability defect of a family $\cal F$ of subsets of $[n]$, $\mbox{cd}^r(\cal F)$, to 
the $S$-disjoint $r$-colorability defect  $\mbox{cd}_S^r(\cal F)$. This was also extended by Abyazi Sani and Alishahi \cite{AA} to the equitable $S$-disjoint $r$-colorability defect $\mbox{ecd}_S^r(\cal F)$ which is defined as follows. Let $\bar{n}=\sum_{i=1}^n s_i$. Then $\mbox{ecd}_S^r({\cal F})$ is defined by
$$\bar{n}-\max{\left\{\sum_{i=1}^r |A_i|\:|\: \{A_1,\dots,A_r\}\:\; \mbox{equitable and}\:\: S\mbox{-disjoint} \:\forall F\in {\cal F}, 1\le i\le r \;\: F\not\subseteq A_i\right \}}.$$
For a subset $P$ of $[n]$, we define the $S$-weight of $P$ to be
$$w_S(P)=\sum_{i\in P} s_i.$$
For a partition ${\cal P}=\{P_1,\dots, P_l\}$ of $[n]$, we also define the $r$-uniform Kneser-type hypergraph $\mbox{KG}_S^r({\cal F},{\cal P})$ to be a hypergraph with the vertex set of those $A\in {\cal F}$ that have at most one element from each $P_1,\dots, P_l$ and the edge set of all multi-sets $\{A_1,\dots, A_r\}$ of the vertices that are $S$-disjoint. We then have the following theorem.
\begin{theorem}
If the partition ${\cal P}=\{P_1,\dots, P_l\}$ has the property that the $S$-weight of each partition piece is at most $r$, then one has
$$\chi(\mbox{KG}_S^r({\cal F},{\cal P}))\ge \left\lceil\frac{\mbox{ecd}^r_S({\cal F})}{r-1}\right\rceil.$$
\end{theorem}

\begin{proof} For each $1\le i\le n$,  we make $s_i$ different copies of $i$, say $(i,1),\dots, (i,s_i)$ and make the set $[n]$ into the bigger set $[\bar{n}]$.  So we have a natural map $f:[\bar{n}]\rightarrow [n]$ that sends any copy of $i$ to $i$. We define the lifted family $\bar{\cal F}$ to be all subsets $A$ of $[\bar{n}]$ such that $f(A)\in \cal F$ and also all two-element subsets of $[\bar{n}]$ with two different copies of the same number. Finally, we define a partition $\bar{{\cal P}}=\{\bar{P_1},\dots, \bar{P_l}\}$ by replacing any element $i$ in a partition piece with all of its $s_i$ copies. Hence $|\bar{P_i}|=w_S(P_i)\le r$. Now we claim that $f$ defines a hypergraph homomorphism from $\mbox{KG}^r(\bar{{\cal F}},\bar{\cal P})$ to $\mbox{KG}_S^r({\cal F},{\cal P})$ and hence
$$\chi(\mbox{KG}_S^r({\cal F},{\cal P}))\ge \chi(\mbox{KG}^r(\bar{{\cal F}},\bar{\cal P})).$$
The proof of the claim is straightforward, notice that the special two-element subsets of $\bar{\cal F}$, do not appear as vertices of this hypergraph. It remains to check that $\mbox{ecd}^r_S({\cal F}) \le \mbox{ecd}^r(\bar{\cal F})$, which will finish the proof of the theorem by applying Theorem \ref{thm1}. If $\{A_1,\dots, A_r\}$ is an equitable disjoint family in $[\bar{n}]$ such that no element of $\bar{\cal F}$ is a subset of one of $A_1,\dots, A_r$, then $f(A_1),\dots, f(A_r)$ is an $S$-disjoint equitable family of subsets of $[\bar{n}]$ with $|f(A_i)|=|A_i|$ (note that because of the special two element subsets in $\bar{\cal F}$, each $A_i$ must contain at most one copy from each element). Also, no $F\in {\cal F}$ is a subset of of one of $f(A_1),\dots, f(A_r)$. This implies that  $\mbox{ecd}^r_S({\cal F}) \le \mbox{ecd}^r(\bar{\cal F})$. The theorem is proved.
\end{proof}

\begin{remark}
When $\cal P$ is the trivial partition of $[n]$ into singletons, this result extends the corresponding inequality
$$\chi(\mbox{KG}^r_S({\cal F}))\ge \left\lceil\frac{\mbox{ecd}^r_S({\cal F})}{r-1}\right\rceil.$$
obtained by Abyazi Sani and Alishahi in \cite{AA} with the extra assumption that $s_i<\mu(r)$, where $\mu(r)$ is the largest prime factor of $r$.
\end{remark}

\section{Examples}
 
In this section, we study the Kneser hypergraph of a special family introduced in \cite{AA} and its generalizations. For integers $n>k>a\ge 0$ and $k>s\ge 0$, define ${\cal H}(n,k,a,s)$ to be the family of all $k$-subsets $F \subseteq [n]$ with $F\not\subseteq_s \{n-a+1,...,n\}$ and let $\mbox{KG}^r(n,k,a,s)$ be the $r$-uniform Kneser hypergraph with the vertex set ${\cal H}(n,k,a,s)$ and the edge set of all $r$-subsets $\{F_1,\dots, F_r\}$ of vertices with pairwise intersection of at most $s$ elements. The case $s=0$, was considered by Abyazi Sani and Alishahi in \cite{AA} and was denoted by $\mbox{KG}^r(n,k,a)$.
\begin{remark}\label{rem}
The pair $({\cal H}(n,k,a,s), {\cal P})$  is $\lfloor s/2 \rfloor$-good, if $n\ge rk$, $|P_i|\le r$, $1\le s<k$ and at least $s+1$ of the non-empty partitions of $\cal P$ have empty intersection with $A=\{n-a+1,n-a+2,\dots, n\}$. The reason is as follows. Let $s'=\lfloor s/2 \rfloor$. Assume for a $\cal P$-admissible subset $B$ and an element $F$ in ${\cal H}(n,k,a,s)$ we have $F\subseteq_{s'} B$. Then there is a subset $E$ of size at most $s'$ such that $F\backslash E\subseteq B$. Since $F$ has at least $s+1$ elements outside of $A$, so $F\backslash E$ has $t\ge s+1-s'$ elements outside of $A$. Assume $P_1,\dots, P_{s+1}$ be the partitions with empty intersection with $A$. If $t\ge s+1$, then one can add arbitrarily elements from different partitions that have empty intersection with $F\backslash E$ so it become a $\cal P$-admissible element $F'\in {\cal H}(n,k,a,s)$ with $F'\subseteq_{s'} B$. If $t<s+1$, then $F\backslash E$ has non-empty intersection with at most $t$ of $P_1,\dots, P_{s+1}$, so we can use elements from those $P_1,\dots, P_{s+1}$ with empty intersection with $F\backslash E$ and if needed other partition parts to complete $F\backslash E$ to a $\cal P$-admissible $F'\in {\cal H}(n,k,a,s)$ such that $F'\subseteq_{s'} B$.

\end{remark}

 The following lemma is an extension of a computation made in \cite{AA} for $\mbox{ecd}^r({\cal H}(n,k,a,0))$.

\begin{lemma}
Let $n,k,r,s$, and $a$ be integers with $k,r \geq 2$ and $n \geq rk$, $0\le s<k$, and $n>a+s$. Then, one has
  \begin{equation*}
\mbox{ecd}^r({\cal H}(n,k,a,s),s)=
 \begin{cases}
 \text{$n-r(k-s-1)$} & 
 \text{$a \leq k-s-1$}\\
 \text{$n-r(k-s-1)-\lfloor \frac{a}{k-s} \rfloor$} & \text{$ k-s \le a \leq r(k-s)-2$}\\
 \text{$n-a$} & \text{$a \geq r(k-s)-1$}
 \end{cases}       
 \end{equation*} 
\end{lemma}

 \begin{proof} Let $A=\{n-a+1,\dots, n\}$. We prove each case separately.
 
 \begin{enumerate}
\item In the first case, ${\cal H}(n,k,a,s)$ is $[n]\choose k$ of all $k$-subsets of $[n]$ and it follows from Remark \ref{rem0}. 
\item In the second case, let $X_0,X_1,...,X_r$ be a partition of $[n]$ such that as in the definition the generalized $r$-colorability defect, there are no $F\in {\cal H}(n,k,a,s)$ such that $F\subseteq_s X_i$ for some $i=1,\dots, r$. We show that $|X_i| \leq k-s$ for $1\leq i\leq r$ and if $|X_i|=k-s$ then $X_i\subseteq A$. Assume that $|X_1| \geq k-s+1$, and since the partition is equitable $|X_i| \geq k-s$ for $1\le i\le r$. Hence, there exist $1\leq i \leq r$ such that $X_i \not\subseteq A$. Let  $X'_i$ be a $k-s$ subset of $X_i$ such that $X'_i \not\subseteq A$. Since $n>a+s$ there exist at least $s$ elements in $[n]\backslash A$ so we can extend $X'_i$ to a k-subset $F$ such that $F\not\subseteq_s A$ and so $F\in {\cal H}(n,k,a,s)$ , $F \subseteq_s X_i$ which violates the assumption on the partition. From the previous argument one can deduce the fact that $|X_i|= k-s$ can only happen when $X_i \subseteq A$. Based on these facts:
$$\mbox{ecd}^r({\cal H}(n,k,a,s),s) \geq n- (k-s) \lfloor \frac{a}{k-s} \rfloor - (k-s-1)(r-\lfloor \frac{a}{k-s} \rfloor) $$
$$= n-r(k-s-1)-\lfloor \frac{a}{k-s} \rfloor$$
This bound is sharp since you can find such an equitable partition by taking $\lfloor \frac{a}{k-s}\rfloor$ disjoint $(k-s)$-subsets of $A$ as $X_1,...,X_{\lfloor \frac{a}{k-s}\rfloor}$  and $r-\lfloor \frac{a}{k-s}\rfloor$ arbitrary disjoint $(k-s-1)$-subsets of the remaining elements as other $X_i$' s.\\
\item In the third case, If $X_0=[n]\backslash A$ and $X_1,\dots, X_r$ be a equitable partition of $A$, then clearly there is no $F\in {\cal H}(n,k,a,s)$ such that $F\subseteq_s X_i$ for some $1\le i\le r$. If $|X_0|<n-a$ then $|X_1\cup\dots\cup X_r|>a$. If $a\ge r(k-s)$, hence at least one $X_i$ has a size of at least $k-s+1$, which is not possible by the argument in the previous step. If $a=r(k-s)-1$, then we must have $|X_i|=k-s$ for all $i$ and hence $X_i\subseteq A$. This is not possible either, because it implies that $a\ge r(k-s)$. So $\mbox{ecd}^r$ is $n-a$.
\end{enumerate}
 \end{proof}
 
\begin{theorem}\label{pre} 
Let $n,k,r,s$ and $a$ be integers with $k,r \geq 2 \ , \ n>a\ge 0$, $n \geq rk$, $0\le s<k$, and $a\le r(k-s-1)$. Then, one has
$$\chi(\mbox{KG}^r(n,k,a,s)) \geq \left\lceil\frac{n-r(k-\lfloor s/2\rfloor -1)}{r-1}\right \rceil.$$
\end{theorem}
\begin{proof}
Let $A=\{n-a+1 , ... , n\}$. Take a partition ${\cal P} =\{ P_1,...,P_l \}$ of $[n]$ such that $|P_i|=r$ for all $1\le i\le k-s-1$ and $|P_i|\le r$ otherwise, and $A \subseteq \bigcup^{k-s-1}_{i=1} P_i$. Now, $\mbox{KG}^r({[n] \choose k},{\cal P},s)$ is a sub-hypergraph of $\mbox{KG}^r(n,k,a,s)$, because if a $k$-subset $F$ is $\cal P$-admissible then it contains at most $(k-s-1)$ elements from $A$ and hence $F\not\subseteq_s A$. The result follows now from Theorem \ref{main'}. Recall that  $\mbox{ecd}^r({[n]\choose k},s)=n-r(k-s-1)$, and by Remark \ref{rem1}, the pair $({[n]\choose k},{\cal P})$ is $\lfloor s/2 \rfloor$-good.
\end{proof}

{\textbf{Remark.}} The above theorem, for the case when $s=0$, was conjectured in \cite{AA}. They showed that it is true when $a\le 2(k-1)$. This was generalized by Aslam, Chen, Coldren, Frick, and Seitanrata in \cite{F} for $a\le b(r)(k-1)$, where $b(r)$ for the prime decomposition $r=2^{\alpha_0}p_1^{\alpha_1}\dots p_m^{\alpha_m}$ is defined to be $2^{\alpha_0}(p_1-1)^{\alpha_1}\dots (p_m-1)^{\alpha_m}$. Our theorem hence, is a generalization of these results.
\\
 
 The following hypergraph is considered in \cite{F}.  Let $\mbox{KG}^r(n,k,{\cal P})_{t-wide}$ be the sub-hypergraph of $\mbox{KG}^r(n,k,{\cal P})$ induced by the vertices that are not contained in any one of the sets $\{i,i+1,...,i+t-1\}$ for $i \in [n-t+1]$. The following theorem is proved in \cite{F}.
 \begin{theorem}
Let $k \geq 1$ be an integer, $r \geq 2$ a prime, and $n \geq rk$ an integer. Let ${\cal P} = \{P_1,...,P_l\}$ be a partition of $[n]$ with $|P_i| \leq r-1$. Let $t \leq r(k-3)+2$. Then  
$$\chi(\mbox{KG}^r(n,k,{\cal P})_{t-wide})=\left\lceil \frac{n-r(k-1)}{r-1} \right\rceil.$$
\end{theorem}
In some special cases we can improve their result. 
 \begin{theorem}
Let $k \geq 1$ be an integer $r \geq 2$, and $n \geq rk$ an integer. Let $t \leq r(k-2)+1$ then  
$$\chi(\mbox{KG}^r(n,k)_{t-wide}) =\left\lceil \frac{n-r(k-1)}{r-1} \right\rceil.$$
\end{theorem}
\begin{proof}
Let $l=\lceil \frac{n}{r} \rceil$ and ${\cal P}=\{P_1,...,P_l\}$ be a partition of $[n]$ such that $$P_i=\{(i-1)r+1 , ... , ir\}$$ for $1\leq i \leq l-1.$
Then, $\mbox{KG}^r(n,k,{\cal P})$ is a sub-hypergraph of $\mbox{KG}^r(n,k)_{t-wide}$, because suppose that a $\cal P$-admissible $k$-subset $F$ is a subset of $X=\{i,i+1,\dots, i+t-1\}$. Then, the smallest value that $t$ can have is when $i\in F$ is the last element of some $P_j$ and $(i+t-1)\in F$ is the first element of $P_{j+k-1}$ and $P_{j+1},\dots, P_{j+k-2}$ are subsets of $X$, that is $t\ge r(k-2)+2$, which contradicts the assumption on $t$. The result then follows from Theorem \ref{thm1} and the standard coloring of the Kneser hypergraph $\mbox{KG}^r(n,k)$.
\end{proof}
\begin{remark}
The family of $t$-wide subsets are very interesting examples to compare the colorability defects for them. It is proved in \cite{F} that the topological $r$-colorability defect of this family for $t\le r(k-3)+2$ is at least $n-r(k-1)$ but we will show that if $n >\max{\{ rt, r(k-1)\}}$ then,
  \begin{equation*}
\mbox{ecd}^r({[n] \choose k}_{t-wide})=
 \begin{cases}
 \text{$n-r(k-1)$} & 
 \text{$t \leq k$}\\
 \text{$n-rt$} & \text{$t > k $}\\
 \end{cases} 
.      
 \end{equation*} 
 Therefore, there exist examples where the topological colorability defect is better than the equitable colorability defect. 
 \end{remark}
\begin{proof}
 In the  first case, the family of $t$-wide $k$-subsets is the same as the family of all $k$-subsets and the result follows by Remark \ref{rem0}. In the second case, let $X_0,X_1,...,X_r$ be a partition of $[n]$ such that no $F\in {[n]\choose k}_{t-wide}$ is a subset of one of $X_1,\dots, X_r$ . Note that for all $1 \leq i \leq r$,  one has  $|X_i| \leq t$ . Because otherwise, take a $k$-subset of $X_i$ that contains its smallest and its biggest elements. This subset is a $t$-wide $k$-subset inside $X_i$ and therefore violates the assumption on the partition. This shows $|X_0| \ge n-rt$. Finally, since the partition given by $X_i=\{(i-1)t+1 , ... , it\}$ for $1 \leq i \leq r$ with $|X_0|=n-rt$ has the property that no $t$-wide $k$-subset is inside one of $X_1,\dots, X_r$, the claim follows.
\end{proof}
It is an interesting problem to see that if it is true that the topological $r$-colorability defect of Frick for a family of subsets is always greater than or equal to the equitable $r$-colorability defect.

\end{document}